\documentclass[12pt]{amsart}
 \usepackage{graphicx}
\usepackage{tikz}        % Tikzpicture
\usepackage{amssymb}
\usepackage{epstopdf}
\usepackage[mathscr]{eucal}
\usepackage{amsmath,amssymb,amscd}
\usepackage{color}
\usepackage{epsfig}
\usepackage{bbold}
\usepackage{hyperref}
\usepackage{MnSymbol}

\newtheorem{theorem}{Theorem}

\newtheorem{definition}{Definition}
\newtheorem{corollary}[theorem]{Corollary}

\newtheorem{proposition}[theorem]{Proposition}
\theoremstyle{definition}
\newtheorem{remark}{Remark}

\def \mb{\mathbb}

\def \mk{\mathfrak}

\def \R{\mb R}                 %%%%% Real Number
\def \C{\mb C}                 %%%%% Complex Number
\def \a{\alpha}         %%%%% alpha
\def \b{\beta}           %%%%% beta
\def \D{\Delta}         %%%%% Delta
\def \th{\theta}       %%%%% \varphi
\def \la{\langle}        %%%%% rightarrow
\def \ra{\rangle}       %%%%% double rightarrow

\newcommand {\diag} {\text{diag}}
\newcommand {\sign} {\text{sign}}

\newcommand {\q} {\mathbf{q}}

\newcommand {\0} {\mathbf{0}}

\newcommand {\cM} {\mathcal{M}}
\newcommand {\cS} {\mathcal{S}}

\newcommand{\EQ}[1]{\begin{equation}\begin{split} #1 \end{split}\end{equation}}

\title{Odd index of the amended potential implies linear instability}
\begin{document}
\maketitle
\markboth{Yanxia Deng,  Shuqiang Zhu}{Odd index of the amended potential implies linear instability}
\author{\begin{center}
		{Yanxia Deng$^1$  and Shuqiang Zhu$^2$}\\
		\ \ \ \ \ \ \ \ \ \ \ \ 	\ \ \ \ \ \ \ \ \ \ \ \ 	\ \ \ \ \ \ \ \ \ \ \ \ 	\ \ \ \ \ \ \ \ \ \ \ \ 		\ \ \ \ \ \ \ \ \ \ \ \ 
		
		\bigskip
       $^1$ School of Mathematics (Zhuhai), Sun Yat-sen University, Zhuhai, Guangdong, China \quad dengyx53@mail.sysu.edu.cn\\
		$^2$School of Mathematical Sciences, University of Science and Technology of China, Hefei,  China \quad zhus@ustc.edu.cn
	\end{center}}

  \begin{abstract}
 For a relative equilibrium of a symmetric simple mechanical system, if the Morse index of the corresponding amended potential is odd, whether the nullity is zero or not,  it is linearly unstable. We also provide a sufficient condition for spectral instability.
  \end{abstract}
  \vspace{2mm}
  
  \textbf{Key Words:}  Linear stability; spectral stability; reduced energy-momentum method; spectral flow; relative Morse index. 
  \vspace{8mm}

\section{Introduction}\label{sec:intro}
A general mechanical system with a Hamiltonian that is the sum of the kinetic and potential energy is said to possess symmetry if the Hamiltonian function is invariant under the action of a group acting on the canonical phase space by canonical transformations. Following \cite{SLM91}, we call it a \emph{symmetric simple mechanical system}. In particular, let $Q$ be a Riemannian \emph{configuration manifold}, and the cotangent bundle $P=T^*Q$ be the \emph{canonical phase space} endowed with the \emph{canonical symplectic two-form}. (cf. \cite{Arn91}\cite{Mar92}\cite{SLM91}).

Let $H(q,p)=\frac{1}{2}\| p\|^2_q +V(q)$ be the Hamiltonian $H: P\to\R$, and suppose there is a group $G$ acting on $Q$ by isometries that leaves $V$ invariant. Here $\|\cdot \|_q$ is the norm induced on $T_q^*Q$. 
Let $\mk g$ (resp. $\mk g^*$) be the Lie algebra (resp. dual of Lie algebra) of $G$. Consider a relative equilibrium $(q_e, p_e)\in T^*Q$ with velocity $\xi\in \mk g$ and momentum value $\mu = \mathbf J(q_e, p_e)\in\mk g^*$, and let $G_\mu$ be the stabilizer of $\mu \in\mk g^*$ for the coadjoint action of $G$ in $\mk g^*$. We refer the reader to section 1 of \cite{SLM91} for a detailed introduction of the relative equilibrium.

A central problem is the stability analysis of the relative equilibrium. For example, the Lyapunov stability of the corresponding equilibrium $z_e$  in the reduced system $\mathbf{J}^{-1}(\mu)/G_\mu$ is guaranteed if it is a local minimum of the \emph{augmented Hamiltonian}, $H_\xi:=H-\la \mathbf{J}, \xi\ra$, where $\la\cdot,\cdot\ra$ is the duality paring between $\mk g^*$ and $\mk g$.

 By the energy-momentum method,  to check the local minimality of $z_e$ for $H_\xi$, it suffices to check the 
positive definiteness of the Hessian $\delta^2 H_\xi$ on a subspace $N$ of  $T_{(q_e, p_e)} \mathbf{J}^{-1}(\mu)$, where $N$ is traversal to the submanifold $G_\mu(q_e, p_e)$. The reduced energy-momentum method introduced in \cite{SLM91}  provides an easy way to check the definiteness of $\delta^2 H_\xi|_{N}$.

	In particular, for each $q\in Q$, let the \emph{locked inertia tensor} be the map $\mathbb I: \mk g\to \mk g^*$ defined by 
\[  \la  \mathbb I \eta, \xi\ra = (\eta_Q(q), \xi_Q(q)),    \]
where $(,)$ is the Riemannian metric on $Q$, $\eta_Q$ is the \emph{infinitesimal generator} on $Q$ induced by $\eta$. The \emph{amended potential} is  a function on $Q$ defined by
\begin{equation*}
V_\mu(q):=V(q)+\frac{1}{2}\la  \mu, \mathbb I^{-1} \mu\ra. 
\end{equation*}	
For any $(q, p)\in \mathbf{J}^{-1}(\mu)$,  the augmented Hamiltonian can be written as 
\[ H_\xi =K_\mu(p) + V_{\mu} (\q),   \]
with $K_\mu(p)$  being the \emph{amended kinetic energy}. 	

By specially chosen coordinates of $\mathbf{J}^{-1}(\mu)$, one has that $\delta^2 H_\xi$ is block-diagonal and $\delta^2 K_\mu$ is positive definite. Thus, the positive definiteness of $\delta^2 H_\xi$ is equivalent to that of $\delta^2 V_\mu$ on a subspace $\mathcal V$ of $T_{q_e}Q$. %The space $\mathcal V$ is the orthogonal complement of $\mk g_\mu(q_e)$ in $T_{q_e} M$.  
In particular, let $\mk g_\mu\cdot q_e$ be the tangent space of the orbit $G_\mu q_e$ in $Q$  at $q_e$, one can take $\mathcal V:=(\mk g_\mu\cdot q_e)^\bot$, the orthogonal complement of  $\mk g_\mu\cdot q_e$ with respect to the Riemannian metric. In section \ref{sec:rmk}, we provide an example of the above formulations to the n-body-type problem. %Under certain conditions, the block $\delta^2 V_\mu$ is also diagonalizable, which will simplify the computation further, but we will not need it for our problem.  

% It is found that there is a coordinate system of $\mathbf{J}^{-1}(\mu)/G_\mu$ (may not be canonical) under which the second variation $\delta^2 H_\xi$ is block-diagonal with the form $\delta^2 K_\mu+\delta^2 V_\mu$, where $\delta^2 K_\mu$ is always positive definite and $V_\mu$ is the \emph{amended potential} defined on some sub-manifold of $Q$. Thus, the positive definiteness of  $H_\xi$ is the same as that of $V_\mu$. 

More generally, when the positive definiteness of $\delta^2 V_\mu$ fails, we consider the linearization of the Hamiltonian vector filed $X_{H_\xi}$ at $z_e$ and study the linear stability of $z_e$. The Hamiltonian vector field is give by $\Omega\delta^2 H_\xi|_N$ with $\Omega$ an invertible skew-symmetric matrix. When $\delta^2 H_\xi|_N$, or equivalently, $\delta^2 V_\mu|_{\mathcal V}$, is invertible, it is known that if the Morse index of $\delta^2 V_\mu|_{\mathcal V}$ is odd then $z_e$ is spectrally unstable, see Proposition \ref{pro:original}. 

The next question is: what can we say about the stability of $z_e$ if $\delta^2 V_\mu$  is not invertible? We confirm in this paper that odd Morse index of $\delta^2 V_\mu$, regardless of the invertibility of $\delta^2 V_\mu$, leads to the linear instability of $z_e$.      To state the main theorem, let's recall the basic notions of the symplectic vector space. 

A symplectic vector space is a  pair $(\R^{2n}, \omega)$, where $\omega$ is a nondegenerate, skew-symmetric, bilinear form. Fixing a basis, the bilinear form $\omega $ is represented by an invertible  skew-symmetric matrix $\Omega$. If $\Omega=J:=\begin{bmatrix}
\0 &-I\\I &\0
\end{bmatrix}$,  the basis is symplectic. 

\begin{definition}
The matrix $JB\in\R^{2n\times 2n}$ is \emph{spectrally stable} if all of its eigenvalues are on the imaginary axis; it is \emph{linearly stable} if it is spectrally stable and semi-simple. It is \emph{linearly unstable} if it is not linearly stable.
\end{definition}

Our main result is:
\begin{theorem}\label{thm:Main}
Let $B\in\R^{2n\times 2n}$ be a symmetric matrix. If the nullity or the Morse index of $B$ is odd, then $JB$ is linearly unstable.
\end{theorem}
{Note that the result also holds if $J$ is replaced by an arbitrary invertible skew-symmetric matrix $\Omega$. There is always an invertible  matrix $Q$ such that $\Omega=QJQ^T$, which leads to 
$\Omega B=Q(JQ^TB Q)Q^{-1}$. Thus  the stability of $\Omega B$ equals that of $JQ^TB Q$. Since $\nu (Q^TB Q )=\nu (B)$ and $\mathcal M (Q^TB Q )=\mathcal M(B)$, where $\nu$ is the nullity and $\mathcal M$ is the Morse index (i.e. number of negative eigenvalues counting multiplicity) of a real symmetric matrix, so we can replace $J$ by $\Omega$ in the above result. 
}

This theorem generalizes the result in \cite{HS09-1} of the relative equilibrium in N-body problem to a symmetric simple mechanical system. A similar result has been claimed in \cite{BJP14}.  %claimed that if the nullity or the Morse index of $B$ is odd, then $\Omega B$ is not linearly stable. 
However, there is a flaw in their proof, see our discussion in Section \ref{sec:rmk}. 
%This short note is to give a solid proof  to their claim. 
We also obtain a sufficient condition for spectral instability, see remark \ref{rk:sp_in}. An application of Theorem \ref{thm:Main} to the symmetric simple mechanical system is given below.
\begin{corollary}\label{cor:Main}
	Let $(q_e, p_e)$ be a relative equilibrium of a symmetric simple mechanical system. If  the nullity or the Morse index of $\delta^2 V_\mu|_{\mathcal V}$ is odd, then the corresponding equilibrium $z_e$ of $\mathbf{J}^{-1}(\mu)/G_\mu$  is linearly unstable.
\end{corollary}

\section{Proof  of the main result}
In this section, we give two proofs of Theorem \ref{thm:Main}. One method uses the spectral flow and relative Morse index, the other method is more elementary.

%\subsection{Notations}

\subsection{The first proof by spectral flow and relative Morse Index}
See  Section 2.1 of \cite{HS09}  for more details and history of spectral flow and the relative Morse Index. Here we only summarize it for $\R^{2n}$ and $\C^{2n}$ with a (Hermitian) inner product. Let $\{A(\th), \, \th\in[0,1]\}$ be a continuous path of self-adjoint matrices. Roughly speaking, the spectral flow of path
$\{A(\th), \, \th\in[0,1]\}$ counts the net change in the number of negative eigenvalues of $A(\th)$ as $\th$ goes from $0$ to $1$, where the enumeration follows from the rule that each negative eigenvalue crossing to the positive axis contributes $+1$ and each positive eigenvalue crossing to the negative axis contributes $-1$, and for each crossing the multiplicity of eigenvalue is counted. In particular, when eigenvalue crossing occurs at $A(\th)$, the operator 

\EQ{\frac{\partial }{\partial\th}A(\th): \, \ker(A(\th)) \to \ker(A(\th))} is called a \emph{crossing operator}, denoting by $Cr[A(\th)]$. An eigenvalue crossing at $A(\th)$ is said to be regular if the null space of $Cr[A(\th)]$ is trivial. 

We denote the signature of $Cr[A(\th)]$ as 
\EQ{\sign Cr[A(\th)]=\dim E_+(Cr[A(\th)])-\dim E_-(Cr[A(\th)]),} where $E_\pm(Cr[A(\th)])$ is the positive and negative definite subspace for $Cr[A(\th)]$. 

Suppose that all crossings are regular. Let $\cS$ be the set of $\th\in[0,1]$ at which the crossing occurs. Then $\cS$ contains only finitely many points. The spectral flow of $\{A(\th), \, \th\in[0,1]\}$ is
\EQ{Sf(\{A(\th), \, \th\in[0,1]\}):=&\sum_{\th\in\cS_*}\sign Cr[A(\th)]-\dim E_-(Cr[A(0)])\\&+\dim E_+(Cr[A(1)]),} where $\cS_*=\cS\setminus\{0,1\}$.

Denote $A=A(0)$ and $\bar{A}=A(1)$, then 
\begin{definition}[Relative Morse Index]
	The relative Morse index of $\bar{A}$ with respect to $A$ is defined to be the spectral flow
	\EQ{\cM(A, \bar{A})=-Sf(\{A(\th), \, \th\in[0,1]\})}
\end{definition}
\begin{remark}
	When $A=0$, $\cM(0, \bar{A})$ is the usual Morse index of $\bar{A}$,  $\cM(\bar{A})$. 
\end{remark}

\begin{remark}
	One has $\cM(A, \bar{A})+\cM(\bar{A}, \tilde{A})=\cM(A, \tilde{A})$. In particular, let $A=0,$ we get \EQ{\cM(\bar{A}, \tilde{A})=\cM(\tilde{A})-\cM(\bar{A}).}
\end{remark}

We use spectral flow and Morse index to prove our main theorem following \cite{HS09-1}. Let $G=\sqrt{-1}J$, which induces the Krein form on $\C^{2n}$. Some useful facts: 
\begin{itemize} 
	\item $\lambda\in\sigma(JB)$ if and only if $\ker(B+\lambda J)\neq 0$.  
	\item $\sqrt{-1}s\in\sigma(JB)$ if and only if $\ker(B+sG)\neq 0$.
	\item $\sqrt{-1}s\in\sigma(JB)$ if and only if $-s\in\sigma(GB)$.
\end{itemize}

\begin{proof} [The first proof of Theorem \ref{thm:Main}] %Without lose of generality, we assume that $\Omega =J$. 
	Assume that $JB$ is linearly stable. Consider self-adjoint matrices $D_s:=B+sG,\, s\in [0,\infty)$. Note that $G: \ker(D_{s})\to \ker(D_{s})$ is the crossing operator for the path $D_s$. 
	
	Since $\ker(D_0)=\ker(B)=\ker(JB)$, semi-simplicity of $JB$ implies that $\dim\ker(D_0)$ is even, the restriction of $G$ on $\ker(D_0)$ is non-degenerate, and its signature on $\ker(D_0)$ is zero. 
	
	There exists $\epsilon>0$ so that there is no crossing for $s\in (0, \epsilon]$, i.e. there are no eigenvalues of $GB$ in $(0, \epsilon]$. For $s$ large enough, $D_s$ is non-degenerate and its signature is the same as that of $G$, which is zero. 
	
	Let $\kappa$ be the number of total multiplicities of eigenvalues of $GB$ on the interval $[\epsilon, \infty)$. By the assumption that $JB$ is linearly stable (in particular, semi-simple), we have
	\EQ{\label{eq:sf1}n=\kappa+\frac{1}{2}\nu(GB)=\kappa+\frac{1}{2}\nu(B).}
	
	At each crossing $\lambda_0\in [\epsilon, \infty)$, that is, $\ker(D_{\lambda_0})\neq0$, $G$ restricted to $\ker(D_{\lambda_0})$ is non-degenerate. Thus \[\sign(G|_{\ker(D_{\lambda_0})})=\dim\ker(D_{\lambda_0})\quad\text{mod}\, 2.\]
	Since the left side is the difference of the dimensions of positive and negative definite subspace of $G$ on $\ker(D_{\lambda_0})\neq0$, and the right side is the sum.
	
	Taking the sum over all crossings $\lambda_0\in[\epsilon, \infty)$ and assuming $s_1>0$ is large enough so that $\ker(D_{s_1})=0$, we have 
	\EQ{\cM(D_0, D_{s_1})&=-\sum_{\lambda_0\geq\epsilon}\sign(G|_{\ker(D_{\lambda_0})})+\dim E_-(G|_{\ker(D_0)})\\&-\dim E_+(G|_{\ker(D_{s_1})})\\&=-\sum_{\lambda_0\geq\epsilon}\sign(G|_{\ker(D_{\lambda_0})})+\frac{1}{2}\nu(D_0)\\&=-\sum_{\lambda_0\geq\epsilon}\dim\ker(D_{\lambda_0})+\frac{1}{2}\nu(B)\quad\text{mod}\, 2\\&=-\kappa+\frac{1}{2}\nu(B)\quad\text{mod}\, 2}
	The last equality again used the semi-simplicity of $JB$. Therefore,
	\EQ{n-\cM(B)&=\cM(D_0, D_{s_1})\\&=-\kappa+\frac{1}{2}\nu(B)\quad\text{mod}\, 2\\&=\nu(B)-n\quad\text{mod}\, 2, \quad \text{by}\, (\ref{eq:sf1})} implying that $\cM(B)$ is even. Contradiction.
\end{proof}

\subsection{The second proof by  elementary method}\label{sec:eleproof}

\begin{proposition}\label{pro:V-invariant}
Let $V:=\ker(B)$.	If $V$ is not $J$ invariant, then $JB$ is not semi-simple. If $V$ is $J$ invariant, then $V$ is even-dimensional. 
\end{proposition}

\begin{proof}
	Let $W$ be the orthogonal complement of $V$ under the Euclidean inner product of $\R^{2n}$. 	Then $B$ is an isomorphism restricted on $W$. 	If $V$ is not $J$ invariant, we assume that  $JV=V_1\oplus W_1$, with $V_1$ (resp. $W_1$) being a subspace of $V$ (resp. $W$) and $W_1\ne \{ 0\}$. 
	Then there exists  some nonzero element $w\in W$ such that   $Bw\in W_1$.  Since $w\notin V$, we have $JBw\ne 0$. Since $J^2=-I$, 
$J^2 V=J(V_1\oplus W_1)=V$, we have $J W_1\subset V$. Then we have $(JB)^2w=JB(JBw)=0$. Therefore, $JB$ is not semi-simple. 
	
	If 	$V$ is $J$ invariant, we have $JV=V$ since $J$ is invertible.  Restricted on $V$, $J$ induces a skew-symmetric, nondegenerate bilinear form. Then $V$ becomes a symplectic subspace, whose dimension must be even. 
\end{proof}

\begin{proposition}\label{pro:original}
	Assume that $B$ is invertible and $\Omega$ is an invertible skew-symmetric matrix. If  $\Omega B$ is spectrally stable, then the Morse index of $B$ is even. 
\end{proposition}

\begin{proof}
The eigenvalues of $\Omega B$ appear in quadruples $\{\lambda, -\lambda, \bar{\lambda}, -\bar{\lambda}\}$  and they have the same algebraic multiplicities \cite{Arn91}.  If $\Omega B$ is spectrally stable and invertible, i.e. all eigenvalues of $\Omega B$ are on the imaginary axis and not equal to zero, then the eigenvalues of $\Omega B$ appear in pairs: $\{\sqrt{-1}a, -\sqrt{-1}a\}$, $a\neq 0$. Thus $\det(\Omega B)>0$. Since $\det(\Omega )>0$, we get $\det(B)>0$. Hence the Morse index of $B$ is even. 
\end{proof}

\begin{remark}
	This argument  is from  Oh  \cite{Oh87}. This result is  also proved in  \cite{BJP14} and \cite{MRT15} by different arguments. 
%	by  Barutello, Jadanza and Portaluri \cite{BJP14} and  Mu\~{n}oz-Lecanda, Rodr\'{i}guez-Olmos  and Teixid\'{o}-Rom\'{a}n \cite{MRT15} by different arguments. 
\end{remark}

%\begin{proposition}\label{pro:spec}
%Let $W$ be a subspace of $\R^{2n}$ that is both $\Omega$ invariant and $B$ invariant. Assume that $B$ is an isomorphism on $W$. 	 If  $\Omega B$ is spectrally stable, then the Morse index of $B$ restricted on $W$ is even. 
%\end{proposition}

 %\begin{proof} The  matrix $B$ induces a nondegenerate symmetric bilinear form on $W$, so $B$ is diagonalizable under some basis of $W$. Denote by  $\tilde B$  the matrix of $B$ under this basis of $W$. Then $\tilde B$ is diagonal and invertible, hence symmetric. Note that $\Omega$ is also an isomorphism on $W$ since 	$\Omega$ is invertible. Hence, $\Omega$  induces a nondegenerate skew-symmetric bilinear form on $W$. Denote by  $\tilde \Omega$  the matrix of $\Omega$ under the above  basis of $W$. 
 	
 %	Then $\Omega B$  restricted on its  invariant subspace  $W$ has matrix 
%$\tilde \Omega \tilde B$. If   $\Omega B$ is spectrally stable,  so is $\tilde \Omega \tilde B$. Then by Proposition \ref{pro:original},  the Morse index of $B$ restricted on $W$ is even. 
%\end{proof}

\begin{proposition}\label{pro:spec}
Let $W$ be a subspace of $\R^{2n}$ that is both $J$ invariant and $B$ invariant. Assume that $B$ is an isomorphism on $W$. 	 If  $J B$ is spectrally stable, then the Morse index of $B$ restricted on $W$ is even. 
\end{proposition}

 \begin{proof}Note that the $J$-invariance of $W$ implies that $W$ is even-dimensional. The matrix $B$ induces a nondegenerate symmetric bilinear form on $W$, so $B$ is diagonalizable under some basis of $W$. Denote by  $\tilde B$  the matrix of $B$ under this basis of $W$. Then $\tilde B$ is diagonal and invertible, hence symmetric. Note that $J$ is also an isomorphism on $W$ since 	$J$ is invertible. Hence, $J$  induces a nondegenerate skew-symmetric bilinear form on $W$. Denote by  $\tilde J$  the matrix of $J$ under the above  basis of $W$. 
 	
 	Then $J B$  restricted on its  invariant subspace  $W$ has matrix 
$\tilde J \tilde B$. If   $J B$ is spectrally stable,  so is $\tilde J \tilde B$. Then by Proposition \ref{pro:original},  the Morse index of $B$ restricted on $W$ is even. 
\end{proof}

\begin{remark}\label{rk:sp_in}
Proposition \ref{pro:spec} is also a sufficient condition for spectral instability. That is, let $W$ be a subspace of $\R^{2n}$ that is both $J$ invariant and $B$ invariant. Assume that $B$ is an isomorphism on $W$. 	 If the Morse index of $B$ restricted on $W$ is odd, then $J B$ is spectrally unstable. 
\end{remark}

\begin{proof} [The second  proof of Theorem \ref{thm:Main}] 
	%Without lose of generality, we assume that $\Omega =J$.  
Assume that $JB$ is linearly stable.  Then $V=\ker(B)$ is $J$ invariant and $JV=V$ by  Proposition \ref{pro:V-invariant}. 
	
	Firstly, the nullity of $B$ is even by Proposition \ref{pro:V-invariant}.  Secondly, 
	denote by $V^\bot$ the skew-orthogonal complement of $V$, i.e., 
	\[ V^\bot=\{ x\in \R^{2n}| x^T J v =0\  for \  any\  v\in V \}.\] That is,  $V^\bot$ is identical to  the set of vectors Euclidean orthogonal to $JV=V$, hence it is just  the Euclidean orthogonal complement of $V$. Then $V^\bot$ is $B$ invariant and  $B$ is an isomorphism on $V^\bot$. Note that  	$V^\bot$  is also $J$ invariant since 
	\[ v^T J J w=-v^T w =0  \]
for any  $v\in V$, and $w\in V^\bot$.	That $JB$ is linearly stable on $\R^{2n}$ implies that it is  linearly stable on the invariant subspace  $V^\bot$.  By Proposition \ref{pro:spec}, we obtain that $\cM(B)=\cM(B|_{V^\bot})$ is even. 
	
\end{proof}

% \begin{proof}
% 	Firstly, note that if $V$ is invariant under a Hamiltonian matrix $H$, then $V^\bot$ is invariant under that Hamiltonian matrix, since  
% 	\[  (v, Hw)= v^T J H w= - v^T H^T J w = -(Hv, w)=0,   \]
%	where $v\in V$, and $w\in V^\bot$.  Especially,  $V$ is invariant under $JB$ and $J$. Then,  
% 	\[  -BV^\bot=JL V^\bot=J V^\bot =V^\bot,  \]
%	i.e., $V^\bot$ is 
%	invariant under $B$. Thus, $B|_{V^\bot} $ is a symmetric operator.  Let  $\{ e_1, ..., e_{2m} \}$ be the eigenvectors of $B$ on $V^\bot$, and $\lambda_i\in \R\setminus\{ 0\}$ are the corresponding eigenvalues, $i=1, ..., 2m$.  Then  $\{  \lambda_1e_1, ..., \lambda_{2m}e_{2m} \}$form a set of basis of $V^\bot$. Let $\tilde J$ be the matrix of $J|_{V^\bot}$ under this basis. That is, 
%	\[  J[ \lambda_1e_1, ..., \lambda_{2m}e_{2m}] = [ \lambda_1e_1, ..., \lambda_{2m}e_{2m}]  \tilde J.  \]
%	Then $\tilde J$ is symplectic. The operator $L=JB$ restricted on $V^\bot$ under the basis $\{e_1, ..., e_{2m}\}$ is 
%	\[   JB[ e_1, ..., e_{2m}] = J[ \lambda_1e_1, ..., \lambda_{2m}e_{2m}] = [ \lambda_1e_1, ..., \lambda_{2m}e_{2m}]  \tilde J =[ e_1, ..., e_{2m}]  \Lambda \tilde J,  \]	
% 	where $\Lambda = \diag\{ \lambda_1, ..., \lambda_{2m}\}.$ So the matrix of $JB$ on $V^\bot$ is $\Lambda \tilde J$. 
	
%	Note that $\det \tilde J=1$, and that $\det \Lambda \tilde J>0$ since $L$ is linearly stable. We conclude that the Morse index of $B$ is even.  		
%  \end{proof}

\section{Examples}\label{sec:rmk}

We illustrate the importance of the $J$ and $B$ invariance in Section \ref{sec:eleproof} by one  example. Consider the matrix $JB$, with 
\[ B=\diag\{ b_1, ..., b_n, b_{n+1}, ..., b_{2n}\}.  \]
Let $\{e_1, ..., e_{2n}\}$ be the standard basis of $\R^{2n}$. Denote by $U_k$ the 2-dimensional subspace spanned by $e_k$ and $e_{n+k}$.  Then $U_k$ is both $J$ and $B$ invariant. Restricted on $U_k$,  $JB$ becomes the matrix 
\[ \begin{bmatrix}
	0&-b_{n+k}\\b_{k}&0
\end{bmatrix}.   \] 
\begin{enumerate}
	\item  If $b_kb_{n+k}=0$ and $b^2_k+b^2_{n+k}\ne0$, i.e., the kernel of $B$ is not $J$ invariant, then the normal form of $JB$ on $U_k$ is  $\begin{bmatrix}
		0&1\\0&0
	\end{bmatrix}.$
	\item If $b_kb_{n+k}>0$, i.e., the Morse index of $B$ on $U_k$ is even,  then the normal form of $JB$ on $U_k$ is 
 $\begin{bmatrix}
	\lambda \sqrt{-1}&0\\0&-\lambda \sqrt{-1}
	\end{bmatrix}, \lambda \in \R, \lambda\ne 0$. 
	\item If $b_kb_{n+k}<0$, i.e., the Morse index of $B$ on $U_k$ is odd,  then the normal form of $JB$ on $U_k$ is  
%$\diag\{ \lambda, -\lambda\},  \lambda \in \R, \lambda\ne 0$.
	 $\begin{bmatrix}
	\lambda&0 \\ 0& -\lambda	\end{bmatrix}, \lambda \in \R, \lambda\ne 0$.
\end{enumerate}

%Our main result reads in \cite{BJP14} in the following form  (Theorem 3.11, Section 3 of \cite{BJP14})
%\begin{center}
%%	``If $JB$ is spectrally stable, then $\cM(B)$ is even. ``
%\end{center}
 %\textbf{ or }
{The proof of Theorem \ref{thm:Main} in \cite{BJP14} relies on  the following result (Theorem 3.11, Section 3 of \cite{BJP14})  
\begin{center}
	``If $JB$ is spectrally stable, then $\cM(B)$ is even. ''
\end{center}}
The quoted statement is not true as can be seen by the above discussions. We can construct examples of spectrally stable $JB$ with  $\cM(B)$ being odd. For instance,
\[  B=\diag\{ -2, -1, 1,  -1, 0, 0\}. \]

Finally, let us give an application of Theorem \ref{thm:Main} to the planar n-body-type problem. Consider  a simple mechanical system with configuration manifold being  \[Q=\{q=(q_1,\cdots, q_n)\in \R^{2n}\setminus\D| \sum_{i=1}^n m_i q_i=\textbf{0}\},\] where $\D$ is the collision set. The Riemannian metric on $Q$ is $(v, w)= v^T M w$, where $ M=\diag\{m_1, m_1, m_2, m_2, ..., m_n, m_n \}. $The kinetic energy  is 
$K=\frac{1}{2} p^T M^{-1}p$,  the potential is $V=-\sum^n_{i,j=1, i< j} \frac{m_im_j}{|q_i-q_j|^\a}$,  $ \a>0$, and the symmetry group is $SO(2)$. The momentum map  $\mathbf J$ is the usual angular momentum.  The Newtonian n-body problem is the case when $\a=1$. The locked inertia tensor is $\mathbb I(q)=q^T M q$. The critical points of the  \emph{augmented potential} $V_\xi= V-\frac{\xi^2}{2} \mathbb I$ are called \emph{central configurations}, which lead to relative equilibria of the n-body-type problem. They are also critical points of $U|_{\hat S}$, where  $U=-V$, 
 $S=\{  q\in Q| \mathbb I(q)=1 \}$, and $\hat S=S/S^1$, see \cite{Moe-A} for detail. 
 % The critical points of $V|_{S}$ correspond to relative equilibria. By Lagrange's multiplier theory, they are also critical points of the \emph{augmented potential}: $V_\xi= V-$
%We give a direct proof of the following result. 

\begin{theorem}
Let $0<\a<2$. 	If the Morse index or the nullity of a central configuration as a critical point of the potential function $U|_{\hat S}$ is odd, then the corresponding relative equilibrium is linearly unstable  on $\mathbf J^{-1}(\mu)/SO(2)$. 
	\end{theorem}
 This result is originally proved for $\a=1$ in \cite{HS09-1} by spectral flow,   and then generalized to cases of $0\le \a$ in  \cite{BJP14}. Our proof is based on the reduced energy-momentum method.

 \begin{proof} Let $q$ be a critical point of the potential function $U|_{\hat S}$. The aforementioned subspace $\mathcal V$ in Section \ref{sec:intro} is the orthogonal complement of the $SO(2)$ orbit of the configuration, i.e., $\mathcal V=T_q \hat S \oplus \text{span} \{ q\}$,  which is $(2n-3)$-dimensional. Recall that $\delta^2 V_\mu=\delta^2 V_\xi +corr$, where the term $corr$ is defined as 
$\langle\mathbb I(q)^{-1} d\mathbb I(q)\cdot \delta q \xi,  d\mathbb I(q)\cdot \delta q \xi \rangle.$  Since $G=SO(2)$, $\mk g\approxeq \R$, $\mathbb I$ can be regarded as a real function. 
\[  d\mathbb I(q)\cdot \delta q \xi = \frac{d\mathbb I(q + \delta q \xi ) }{d t}|_{t=0} =2(Mq)^T \delta q \xi. \] 
So
\[  \langle\mathbb I(q)^{-1} d\mathbb I(q)\cdot \delta q \xi,  d\mathbb I(q)\cdot \delta q \xi \rangle = \frac{4\xi^2}{\mathbb I(q)} \delta q^TMq (M q)^T \delta q,\] 
and the matrix of the bilinear form $corr$ is the projection matrix $\frac{4\xi^2}{\mathbb I(q)}  Mq (M q)^T$. 
  Thus, the matrix of $\delta^2 V_\mu$ on $\mathcal V$ is simply 
 \[  -M^{-1}D^2 U -\xi^2 I + 4\xi^2 q (M q)^T. \]
The last term  is zero on $T_q \hat S$ and  $4\xi^2 q (M q)^T q= 4\xi^2 \mathbb I (q)  q$.

{Notice that $U$ is homogeneous with degree $-\a$, we have $ DU q=-\a U$ and $ D^2 U q+ DU =-\a DU.$ Since $DU=-\xi^2 Mq$ at the relative equilibrium $q$, we get $ M^{-1}D^2 U q = (\a +1)\xi^2 q.$ Then 
\begin{align*}
&\delta^2 V_\mu q = -(\a+1) \xi^2 q-\xi^2 q+4\xi^2 q=(2-\a)\xi^2 q,\\
&\delta^2 V_\mu|_{T_q \hat S}=  - \delta^2 U|_{\hat S}.
\end{align*}

Thus under the assumption that $0<\a<2$, we have \begin{align*}
&\nu (\delta^2 V_\mu|_\mathcal{V})  = \nu (\delta^2 U|_{\hat S}), \\
& \mathcal M (\delta^2 V_\mu|_\mathcal{V})  =2n-4- \nu (\delta^2U|_{\hat S}) -\mathcal M(\delta^2 U|_{\hat S}).
\end{align*}
}
Hence, that $\mathcal M(\delta^2 U|_{\hat S})$ or $\nu( \delta^2 U|_{\hat S})$ is odd implies that $\mathcal M(\delta^2 V_\mu|_\mathcal V)$ or $\nu(\delta^2 V_\mu|_\mathcal V)$ is odd; then by Theorem \ref{thm:Main} the relative equilibrium is linearly unstable. 

\end{proof}

We remark that for n-body type problem with potential  $V=-\sum^n_{i,j=1, i< j} \frac{m_im_j}{|q_i-q_j|^\a}$,  $ \a\ge 2$, a relative equilibrium is always linearly unstable on $\mathbf J^{-1}(\mu)/SO(2)$. There is a $4$-dimensional subspace  of $T_{(p_e, q_e)} T^*Q$, 
$$E_{1}=span\{ (q, 0),(0, Mq),( q^\bot,0),(0, M q^\bot)\}, \ q^\bot=(-y_1, x_1, ..., -y_n ,x_n),  $$ 
  (cf. Section 5.2  of  \cite{BJP14}), the linearization of the Hamiltonian vector field on $E_1$ has the matrix \[\left(
\begin{array}{cccc}
0 & -\xi  & 1 & 0 \\
\xi  & 0 & 0 & 1 \\
(\a+1) \xi ^2 & 0 & 0 & -\xi  \\
0 & -\xi ^2 & \xi  & 0 \\
\end{array}
\right).\]  
The  eigenvalues are 
$ \left\{0,0,-\sqrt{\a-2} \xi ,\sqrt{\a-2} \xi \right\}$. One can check that if $\a=2$, this matrix is similar to $\left(
\begin{array}{cccc}
0 & 1 & 0 & 0 \\
0 & 0 & 1 & 0 \\
0 & 0 & 0 & 1 \\
0 & 0 & 0 & 0 \\
\end{array}
\right)$.  If $\a>2$, there are positive   eigenvalues.

However, this instability is  from the  invariance of the system  under the space-time scaling transform, $(q, t)\mapsto (\tilde{q},s)$ with 
\[  \tilde{q}=\lambda q, s=\lambda^b t, \quad b=\frac{2+\a}{2}.\]
%Thus when $\a\geq2$, the relative equilibrium of the planar n-body-type problem is always linearly unstable on $ \mathbf J^{-1}(\mu)/SO(2)$. 
Hence, it is a common  practice that we study the linear stability on $E_2$, the symplectic complement of $E_1$, see \cite{Moe-A}, which is a $(4n-8)$-dimensional subspace of $T_{z_e}J^{-1}(\mu)/SO(2)$. On $E_2$, the linearized Hamiltonian vector field is $\Omega (\delta^2 V_\mu + \delta^2 K_\mu)$, with $\delta^2 V_\mu$ restricted on $T_{q_e}  \hat S$.  Thus, we reestablish the following result of \cite{BJP14}.  
	\begin{corollary}
Let $0<\a$. 	If the Morse index or the nullity of a central configuration as a critical point of the potential function $U|_{\hat S}$ is odd, then the corresponding relative equilibrium is linearly unstable on $E_2$. 
	\end{corollary}

 \section*{acknowledgments}
%We would like to thank Cristina Stoica for sending us several preprints, and Shanzhong Sun for helpful references on spectral flow and relative Morse index. 
We would like to thank Cristina Stoica and Shanzhong Sun for discussions and careful readings of the manuscript. 
Shuqiang Zhu would like to acknowledge NSFC(No.11801537) and China Scholarship Council (CSC NO. 201806345013), and 
Instituto Tecnol\'{o}gico Aut\'{o}nomo de M\'{e}xico for their warm hospitality.


\begin{thebibliography}{FF}
\bibitem{Arn91} V.I. Arnold,   Mathematical  Methods  of  Classical  Mechanics,  Translated by K.Vogtmann and A.Weinstein,  Springer-Verlag,  New York,  1991.	
	
		\bibitem{BJP14}V.L. Barutello,  R.D. Jadanza,  A. Portaluri, 
	Linear instability of relative equilibria for n-body problems in the plane, 
	J. Differential Equations 257 (2014), no. 6, 1773-1813.
	



	
	\bibitem{HS09}X. Hu, S. Sun, 
Index and stability of symmetric periodic orbits in Hamiltonian systems with application to figure-eight orbit, 
Comm. Math. Phys. 290 (2009), no. 2, 737-777.

	\bibitem{HS09-1}X. Hu, S. Sun, 
Stability of relative equilibria and Morse index of central configurations, 
C. R. Math. Acad. Sci. Paris 347 (2009), no. 21-22, 1309-1312.

\bibitem{Mar92} J. Marsden, 
Lectures on Mechanics. 
London Mathematical Society Lecture Note Series, 174. Cambridge University Press, Cambridge, 1992. 

\bibitem{Moe-A} R.\ Moeckel,  Celestial Mechanics--Especially   Central Configurations, unpublished lecture notes:

http://www.math.umn.edu/\~{}rmoeckel/notes/CMNotes.pdf

%\bibitem{MHO09} K.R. Meyer, G.R. Hall, D.Offin,  Introduction to Hamiltonian Dynamical Systems and the $N$-Body Problem, Second edition. Applied Mathematical Sciences, 90. Springer, New York, 2009. 

\bibitem{MRT15} M. Mu\~{n}oz-Lecanda, M. Rodr\'{i}guez-Olmos, M. Teixid\'{o}-Rom\'{a}n,   A Hamiltonian study of the stability and bifurcations for the satellite problem.  J. Nonlinear Sci.,  25 (2015), 1347-1390. 


%\bibitem{Rod05} M. Rodr\'{i}guez-Olmos,   Stability of relative equilibria with singular momentum values in simple mechanical systems.  Nonlinearity 19 (2006), no.4, 853-877. 
	\bibitem{Oh87} Y-G  Oh, 
A stability criterion for Hamiltonian systems with symmetry, 
J. Geom. Phys. 4 (1987), no. 2, 163-182. 
	
	\bibitem{SLM91} J.C. 	Simo, D. Lewis, J.E.  Marsden, 
Stability of relative equilibria. I. The reduced energy-momentum method.
Arch. Rational Mech. Anal. 115 (1991), no. 1, 15-59.

%\bibitem{SStoi06}T. Schmah, C.  Stoica, Stability for Lagrangian relative equilibria of three-point-mass systems.  J. Phys. A 39 (2006), no. 46, 14405-14425.



	

	


%   \bibitem{Moe94}R. Moeckel, Linear stability of relative equilibria with a dominant mass,  J. Dynam. Differential Equations 6 (1994), no.1,  37-51.





% \bibitem{Rob99-1} G.E. \ Roberts,  Spectral instability of relative equilibria in the planar n-body problem, Nonlinearity 12 (1999), no. 4, 757-769. 





\end{thebibliography}
\end{document}